\documentclass{article}

\usepackage[a4paper,outer=3cm]{geometry}

\usepackage[utf8]{inputenc}
\usepackage[T1]{fontenc}

\usepackage{times}
\usepackage{comment}
\usepackage{xspace}

\usepackage{longtable}
\usepackage{amsmath,amsthm,amssymb,amsthm}
\usepackage{enumerate}

\usepackage{algorithm}
\usepackage{algpseudocode}
\algnewcommand\algorithmicinput{\textbf{Input:}}
\algnewcommand\algorithmicoutput{\textbf{Output:}}
\algnewcommand\Input{\item[\algorithmicinput]}
\algnewcommand\Output{\item[\algorithmicoutput]}

\usepackage{hyperref}

\usepackage{caption}
\usepackage{subfigure}
\usepackage{tikz}
\usepackage{svg}
\usepackage{cleveref}
\usepackage{authblk}
\usepackage{multirow}
\usepackage{graphicx}
\usepackage{wrapfig}
\usepackage{bbding}

\usepackage{pdftricks}
\begin{psinputs}
	\usepackage{pstricks,pst-plot,pst-node,pst-func}
\end{psinputs}
\usepackage{graphics,graphicx}

\tikzstyle{vertex}=[circle, draw, inner sep=0pt, minimum size=6pt]

\usetikzlibrary{positioning,fit,shapes,backgrounds}
\usetikzlibrary{calc}
\usetikzlibrary{backgrounds}

\pgfdeclarelayer{edgelayer}
\pgfdeclarelayer{nodelayer}
\pgfsetlayers{background,edgelayer,nodelayer,main}

\tikzstyle{none}=[]
\tikzstyle{Vertex}=[fill={rgb,255: red,142; green,142; blue,142}, draw=black, shape=rectangle]
\tikzstyle{white vertex}=[fill=white, draw=black, shape=circle]
\tikzstyle{grey vertex}=[fill={rgb,255: red,142; green,142; blue,142}, draw=black, shape=circle]
\tikzstyle{black vertex}=[fill=black, draw=black, shape=circle]
\tikzstyle{small black}=[fill=black, draw=black, shape=circle, scale=.5pt]
\tikzstyle{small gray}=[fill={rgb,255: red,142; green,142; blue,142}, draw=black, shape=circle, scale=.5pt]
\tikzstyle{small white}=[fill=white, draw=black, shape=circle, scale=.5pt]

\tikzstyle{double blue edge}=[->, draw={rgb,255: red,6; green,118; blue,255}, latex-latex]
\tikzstyle{new edge style 0}=[->, draw={rgb,255: red,6; green,118; blue,255}, -latex]
\tikzstyle{Edge}=[-]
\tikzstyle{edge red}=[-, fill=none, draw=red]
\tikzstyle{blue_edge}=[-, draw={rgb,255: red,6; green,118; blue,255}]
\tikzstyle{gray_region}=[-, fill={rgb,255: red,220; green,220; blue,220}]

 \usepackage{marginnote}

\newcommand{\m}[1]{}

\usepackage{amsthm, amsmath, amsfonts, amssymb}
\usepackage{thmtools}

\declaretheorem{theorem}

\declaretheorem[numberlike=theorem]{proposition}
\declaretheorem[numberlike=theorem]{corollary}
\declaretheorem[numberlike=theorem]{conjecture}

\declaretheorem[numbered=no,name=Gy\'arf\'as-Sumner conjecture]{GSconjecture}

\declaretheorem[numberlike=theorem]{lemma}

\declaretheorem{claim}

\declaretheorem[numberlike=theorem]{definition}
\declaretheorem[numberlike=theorem]{observation}
\declaretheorem[numbered=no,name=Gyárfás--Sumner Conjecture]{conjecture*}

\newcommand{\Gya}{Gy\'arf\'as\xspace}

\newcommand{\Erd}{Erd\H os\xspace}

\definecolor{darkgreen}{RGB}{1, 50, 32}

\newcommand{\mcal}[1]{\mathcal{#1}}

\newcommand{\mF}{\mathcal{F}}

\newcommand{\set}[1]{\left\{#1\right\}}

\newcommand{\ra}{\rightarrow}

\DeclareMathOperator{\Forb}{Forb}

\newcommand{\adj}[1][]{\stackrel{#1}{\sim}}

\newcommand{\F}[1]{\Forb_{ind}{#1}}

\newcommand{\indsub}{\underset{\!\!ind\!\!}{\subset}}

\title{Extension of the \Gya-Sumner conjecture to signed graphs}

\author[1]{Guillaume Aubian}
\author[2]{Allen Ibiapina}  
\author[3]{Luis Kuffner} 
\author[3,4]{Reza Naserasr}
\author[3]{\\ Cyril Pujol \Envelope}
\author[3]{Cléophée Robin}
\author[3,4]{Huan Zhou}

 \affil[1]{\small Université Paris-Panthéon-Assas, CRED Paris, France {Email: \texttt{guillaume.aubian@assas-universite.fr}}}
 
 \affil[2]{\small Universidade Federal do Ceará, Brazil {Email: \texttt{allen.ibiapina@alu.ufc.br}}}
 
\affil[3]{\small Université Paris Cité, CNRS, IRIF, F-75013, Paris, France.\\
	
	 {Emails: \texttt{\{kuffner, reza, cpujol, crobin, zhou\}@irif.fr}}}

\affil[4]{\small Zhejiang Normal University, Jinhua, China. {Email: \texttt{huanzhou@zjnu.edu.cn}}}

\begin{document}
	
	\maketitle
	
	\begin{abstract}
		The balanced chromatic number of a signed graph $\widehat{G}$ is the minimum number of balanced sets that cover all vertices of $G$. Studying structural conditions which implies bound on the balanced chromatic number of signed graphs then is among the most fundamental problems in graph theory. In this work, we initiate the study of coloring hereditary classes of signed graphs. More precisely we say that a set $F=\{\widehat{F_1}, \widehat{F_2}, \ldots, \widehat{F_l}\}$ is a GS (for \Gya-Sumner) set if there exists a constant $c$ such that signed graphs with no induced subgraph switching equivalent to a member of $F$ admit a balanced $c$-coloring. The focus of this work then is to study GS sets of order 2. We show that if $F$ is a GS set of order 2, then $F_1$ is either $\widehat{(K_3,-)}$ or $\widehat{(K_4,-)}$ and $F_2$ is a linear forest. In the case of $F_1=\widehat{(K_3,-)}$ we show that any choice of a linear forest for $F_2$ works. In the case of $F_1=\widehat{(K_4,-)}$ we show that if each connected component of $F_2$ is a path of length at most 4, then $\{F_1,F_2\}$ is a GS set.
		
	\end{abstract}
	
	
	\section{Introduction}
	
	One of the key questions in the theory of proper coloring of graphs is: what structure conditions impose upper bounds on the chromatic number? After a series of constructions of triangle-free graphs of arbitrarily large chromatic number, P. \Erd \cite{Erdos59}, in one of the earliest use of probabilistic methods, proved the following. 
	\begin{theorem}\label{thm:large_girth}
		There exist graphs of arbitrarily large chromatic number and girth. 
	\end{theorem}
	
	This immediately implies that if given a finite set $F$ of graphs, the class of graphs with no induced subgraph isomorphic to a graph in $F$, denoted $\F{(F)}$, has a bounded chromatic number, then $F$ must contain at least one forest. Considering the class of complete graphs, any such set $F$ must also contain a complete graph. \Gya and Sumner, independently, conjectured, in \cite{Gya75} and \cite{Sumner81}, that any such pair is enough for $\F{(F)}$ to have a bounded chromatic number: 
	
	\begin{GSconjecture}\label{conj:Gyarfash-Sumner}
		For any forest $F$ and any complete graph $K_t$ the class $\F{\{F, K_t\}}$ has a bounded chromatic number.
	\end{GSconjecture}

	The goal of this work is to consider potential extensions of this conjecture and related results to the class of signed graphs. To this end, we first introduce the notions and concepts.
	
	\subsection{Definitions and Notations}
	
	A \emph{signed graph} is a pair $(G,\sigma)$ where $G$ is a graph and $\sigma: E(G)\to \{+,-\}$ is a mapping that assigns to each edge one of the two signs: positive or negative. If $\Sigma$ is the set of negative edges, then $(G, \sigma)$ can be equivalently presented as $(G,\Sigma)$. When $\sigma$ is of little importance, we may write $\widehat{G}$ in the place of $(G,\sigma)$. The subgraph $(V(G), \Sigma)$ of $\widehat{G}$ is denoted by $\widehat{G}^{-}$.

	
	The graph $G$ is the \textit{underlying graph} of $\widehat{G}$. A pair of parallel edges of different signs is called a {\it digon}. However, unless especially mentioned, in the rest of this work, we only consider signed simple graphs.
	
	\vspace{2ex}
	
	{\it Switching} a vertex $v$ of $(G,\sigma)$ is to multiply the sign of each edge incident to $v$ by $-$. Observe that switching is an involution, and that the order doesn't matter when switching multiple vertices. Therefore, we may switch a set of vertices, meaning that we switch all of the vertices of the set in an arbitrary order. If $(G,\sigma')$ is obtained from $(G,\sigma)$ by switching some vertices then we say they are {\it switching equivalent} (see \Cref{fig:switch}).
	
	The \textit{sign of a structure} $W$ in $(G, \sigma)$, denoted $\sigma(W)$, is the product of the signs of its edges, considering multiplicity. It is immediate that the sign of a closed walk, and in particular a cycle, will not change after switching.

	A signed graph $\widehat{G}$ is said to be \textit{balanced} if every cycle of it is positive. 
	Hence, we set the following definition: 
	
	\begin{definition}
		A \textit{balanced set} of a sign graph $\widehat{G}$ is a subset of vertices $X\subseteq V(\widehat{G})$ such that every cycle in $\widehat{G}[X]$ is positive. \end{definition}
	
	The signed graph on $G$ where all edges are negative (respectively, positive) is denoted by $(G, -)$ (respectively, $(G,+)$). A signed graph switching equivalent to $(G, -)$ will be denoted by $\widehat{(G, -)}$.
	
	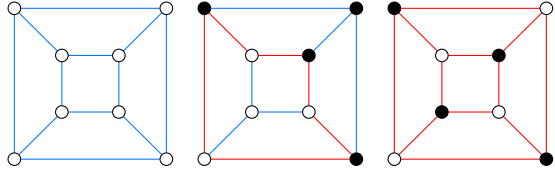
\begin{figure}[h]
		\centering
		\begin{tikzpicture}[scale = .5]
	\begin{pgfonlayer}{nodelayer}
		\node [style=small white] (0) at (-5.75, 0.75) {};
		\node [style=small white] (1) at (-4.25, 0.75) {};
		\node [style=small white] (2) at (-5.75, -0.75) {};
		\node [style=small white] (3) at (-4.25, -0.75) {};
		\node [style=small white] (4) at (-7, 2) {};
		\node [style=small white] (5) at (-3, 2) {};
		\node [style=small white] (6) at (-3, -2) {};
		\node [style=small white] (7) at (-7, -2) {};
		\node [style=small white] (8) at (-0.75, 0.75) {};
		\node [style=small black] (9) at (0.75, 0.75) {};
		\node [style=small white] (10) at (-0.75, -0.75) {};
		\node [style=small white] (11) at (0.75, -0.75) {};
		\node [style=small black] (12) at (-2, 2) {};
		\node [style=small black] (13) at (2, 2) {};
		\node [style=small black] (14) at (2, -2) {};
		\node [style=small white] (15) at (-2, -2) {};
		\node [style=small white] (16) at (4.25, 0.75) {};
		\node [style=small black] (17) at (5.75, 0.75) {};
		\node [style=small black] (18) at (4.25, -0.75) {};
		\node [style=small white] (19) at (5.75, -0.75) {};
		\node [style=small black] (20) at (3, 2) {};
		\node [style=small white] (21) at (7, 2) {};
		\node [style=small black] (22) at (7, -2) {};
		\node [style=small white] (23) at (3, -2) {};
	\end{pgfonlayer}
	\begin{pgfonlayer}{edgelayer}
		\draw [style={blue_edge}] (4) to (0);
		\draw [style={blue_edge}] (0) to (2);
		\draw [style={blue_edge}] (2) to (3);
		\draw [style={blue_edge}] (3) to (1);
		\draw [style={blue_edge}] (1) to (0);
		\draw [style={blue_edge}] (4) to (5);
		\draw [style={blue_edge}] (5) to (1);
		\draw [style={blue_edge}] (3) to (6);
		\draw [style={blue_edge}] (6) to (5);
		\draw [style={blue_edge}] (4) to (7);
		\draw [style={blue_edge}] (7) to (2);
		\draw [style={blue_edge}] (7) to (6);
		\draw [style=edge red] (12) to (8);
		\draw [style={blue_edge}] (8) to (10);
		\draw [style={blue_edge}] (10) to (11);
		\draw [style=edge red] (11) to (9);
		\draw [style=edge red] (9) to (8);
		\draw [style={blue_edge}] (12) to (13);
		\draw [style={blue_edge}] (13) to (9);
		\draw [style=edge red] (11) to (14);
		\draw [style={blue_edge}] (14) to (13);
		\draw [style=edge red] (12) to (15);
		\draw [style={blue_edge}] (15) to (10);
		\draw [style=edge red] (15) to (14);
		\draw [style=edge red] (20) to (16);
		\draw [style=edge red] (16) to (18);
		\draw [style=edge red] (18) to (19);
		\draw [style=edge red] (19) to (17);
		\draw [style=edge red] (17) to (16);
		\draw [style=edge red] (20) to (21);
		\draw [style=edge red] (21) to (17);
		\draw [style=edge red] (19) to (22);
		\draw [style=edge red] (22) to (21);
		\draw [style=edge red] (20) to (23);
		\draw [style=edge red] (23) to (18);
		\draw [style=edge red] (23) to (22);
	\end{pgfonlayer}
\end{tikzpicture} 
		\caption{$(Q_3,+) \simeq (Q_3,\sigma) \simeq (Q_3,-)$}
		Switching black vertices in $(Q_3,\sigma)$, or in $(Q_3,-)$, results in $(Q_3,+)$.
		\label{fig:switch}
	\end{figure}
	
	The following lemma of Zaslavsky \cite{zaslavsky82}, extending a special case, first proven by Harary \cite{Harary53}, characterizes switching equivalent classes of signed graphs. 
	
	\begin{lemma}\label{lem:Zaslavsky}
		Two signed graphs $(G, \sigma)$ and $(G,\sigma')$ are switching equivalent ($\simeq$) if and only if $\sigma(C)=\sigma(C')$ for every cycle $C$ of $G$. 
	\end{lemma}

	\begin{lemma} \label{lem:balanced}
		The following are equivalent:
		\begin{itemize}
			\item $\widehat{G}$ is balanced,
			\item $\widehat{G}$ is switching equivalent to $(G,+)$,
			\item the negative edges of $\widehat{G}$ form an edge-cut $(X, \overline{X})$ of $G$.
		\end{itemize}
	\end{lemma}

	\vspace{2ex}
	We then have a notion of balanced chromatic number for signed graphs. 
	
	\begin{definition}
		The balanced chromatic number of a signed graph $\widehat{G}$, denoted $\chi_b(\widehat{G})$, is the minimum number of balanced sets needed to cover $V(\widehat{G})$.
	\end{definition}
	
	This parameter was first introduced by Zaslavsky \cite{Zaslavsky87} under the name of \textit{balanced partition number}.
	Zaslavsky has also introduced the notion of \textit{$0$-free $p$-coloring of signed graphs}. That is, a coloring $c$ of $V(G)$ with colors from the set $\set{\pm1, \pm2, \ldots, \pm p}$ such that $c(x) \neq \sigma(xy)c(y)$ for each edge $xy$ of $(G,\sigma)$. It can be easily verified that a balanced $p$-coloring of $(G,\sigma)$ is equivalent to a $0$-free $p$-coloring of $(G,-\sigma)$. This can also be viewed as a homomorphism to the signed graph on $p$ vertices where there are both positive and negative edge between each pair of vertices and there is positive loop on each vertex. For more on homomorphisms of signed graphs we refer to \cite{NRS15}.

	The notion naturally extends to a family $\mathcal{G}$ of signed graphs by
	$$\chi_b(\mathcal{G}) = \max\limits_{\widehat{G} \in \mathcal{G}} \chi_b(\widehat{G}),$$ where  ${\chi_b(\mathcal{G}) =\infty}$ if the maximum does not exit.  
	
	\vspace{2ex}
	
	From here on, when we refer to a coloring of a graph, it will be a proper coloring. A coloring of a signed graph, on the other hand, will be a balanced coloring, which could be far from a proper coloring of the underlying graph.
	However, there is a tight relation between the balanced chromatic number of a signed graph and the chromatic number of its negative subgraph \cite{Zaslavsky87}:
	
	\begin{lemma}\label{lem:no_switch}
		$\chi_b(\widehat{G}) \leq \chi(\widehat{G}^{-}) \leq 2\chi_b(\widehat{G})$
	\end{lemma}
	\begin{proof}
		Let $c = \chi_b(\widehat G)$ and consider a partition of $\widehat{G}$ into $c$ balanced subgraphs: $(\widehat{G}_{i})_{i\leq c}$. By \Cref{lem:balanced}, in each $\widehat{G}_i$, all negative edges form an edge cut. In other words, $\widehat{G}_i^-$ is $2$-colorable. Thus, $\widehat{G}^{-}$ is $2c$-colorable.
		
		Conversely, any partition of $G_\sigma^-$ into independent sets induces a partition of $\widehat{G}$ into balanced sets. Therefore, $\chi_b(\widehat{G}) \leq \chi(\widehat{G}^{-})$.
	\end{proof}
	
	This directly implies the following:
	
	\begin{corollary}\label{cor:link_signed_unsigned}
		A family $\mathcal{G}$ of signed graphs has bounded balanced chromatic number if and only if ${\{\widehat{G}^{-} \mid \widehat{G} \in \mathcal{G}\}}$ has bounded chromatic number.
	\end{corollary}
	
	\vspace{2ex}
	\begin{definition}
		A signed graph $\widehat{H}$ is said to be an induced subgraph of a signed graph $\widehat{G}$, denoted $\widehat{H} \indsub \widehat{G}$, if $\widehat{H}$ is isomorphic to a subgraph of $\widehat{G}$ obtained from $\widehat{G}$ by applying the following (commutative) operations:
		\begin{itemize}
			\item removal of a vertex (and all edges incident to it),
			\item switching at vertex.
		\end{itemize}
	\end{definition}
	
	For instance, in \Cref{fig:switch} one can see that $(C_6,-) \indsub (Q_3,+)$.

	A family $\mathcal{G}$ of signed graphs is said to be \textit{hereditary} if it is closed under taking induced subgraphs. Given a family $\mathcal{F}$ of signed graphs, we denote by $\F(\mF)$ the class $\{\widehat{G} \mid \forall \widehat{H} \in \mathcal{F}, \widehat H \not\indsub \widehat{G}\}$. Note that $\F(\mathcal{F})$ is hereditary.
	
	In this paper, we sometimes need to work with signed graphs where switching will not be considered. In such cases, considering negative edges as red and positive edges as blue, we rather refer to the signed graph in hand as a 2-edge-colored graph. We may also prefer to work with the underlying graph of a signed graph in certain cases (where forbidding a subgraph is equivalent to forbidding all possible signatures). To capture the three notions together (graphs, 2-edge-colored graphs and signed graphs) we adopt the following notation. Given $\mF=\set{F_1, (F_2, \pi), \widehat{F}_3}$, where $F_1$ is a graph, $(F_2,\pi)$ a 2-edge-colored graph (with $\pi$ the edge coloring function) and $\widehat{F}_3$ is a signed graph, the class $\Forb_{ind}(\mF)$ is the class of signed graphs $(G, \sigma)$ where no induced subgraph of $G$ is isomorphic to $F_1$, no induced subgraph of $(G, \sigma)$ (as a 2-edge-colored graph) is isomorphic to $(F_2, \pi)$ and no induced subgraph of $(G, \sigma)$ is isomorphic to $\widehat{F}_{3}$ (where switching is permitted). In \Cref{fig:forb} we have an example of $(K_5, \sigma)$ which is in  $\F\{P_3,(K_3,+),\widehat{(K_4,-)}\}$.
	
	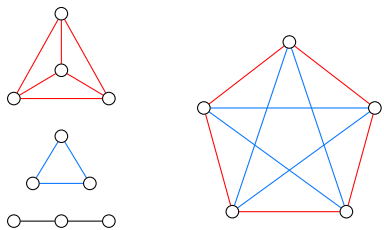
\begin{figure}[h]
		\centering
		\begin{tikzpicture}[scale = .5]
	\begin{pgfonlayer}{nodelayer}
		\node [style=small white] (0) at (0, 2.75) {};
		\node [style=small white] (1) at (-2.25, 1) {};
		\node [style=small white] (2) at (2.25, 1) {};
		\node [style=small white] (3) at (-1.5, -1.75) {};
		\node [style=small white] (4) at (1.5, -1.75) {};
		\node [style=small white] (5) at (-7.25, 1.25) {};
		\node [style=small white] (6) at (-6, 3.5) {};
		\node [style=small white] (7) at (-4.75, 1.25) {};
		\node [style=small white] (9) at (-6.75, -1) {};
		\node [style=small white] (10) at (-5.25, -1) {};
		\node [style=small white] (11) at (-6, 0.25) {};
		\node [style=small white] (12) at (-6, -2) {};
		\node [style=small white] (13) at (-4.75, -2) {};
		\node [style=small white] (14) at (-7.25, -2) {};
		\node [style=small white] (15) at (-6, 2) {};
	\end{pgfonlayer}
	\begin{pgfonlayer}{edgelayer}
		\draw [style=edge red] (0) to (2);
		\draw [style=edge red] (2) to (4);
		\draw [style=edge red] (4) to (3);
		\draw [style=edge red] (3) to (1);
		\draw [style=edge red] (1) to (0);
		\draw [style={blue_edge}] (0) to (3);
		\draw [style={blue_edge}] (3) to (2);
		\draw [style={blue_edge}] (2) to (1);
		\draw [style={blue_edge}] (0) to (4);
		\draw [style={blue_edge}] (4) to (1);
		\draw [style=edge red] (6) to (5);
		\draw [style=edge red] (5) to (7);
		\draw [style=edge red] (7) to (6);
		\draw [style={blue_edge}] (9) to (11);
		\draw [style={blue_edge}] (11) to (10);
		\draw [style={blue_edge}] (10) to (9);
		\draw (14) to (12);
		\draw (12) to (13);
		\draw [style=edge red] (6) to (15);
		\draw [style=edge red] (15) to (7);
		\draw [style=edge red] (15) to (5);
	\end{pgfonlayer}
\end{tikzpicture}  
		\caption{A graph in $\F\{P_3,(K_3,+),\widehat{(K_4,-)}\}$}
		\label{fig:forb}
	\end{figure}

	\begin{definition}
		A finite set $\mathcal{F}$ of signed graphs is a Gyárfás-Sumner set (or $GS$ set) whenever $\chi_{b}(\F(\mathcal{F}))$ is bounded.
	\end{definition}
	
	By \Cref{lem:no_switch}, to determine if $\{\chi_{b}(\widehat{G})\mid \widehat{G} \in \mathcal{G}\}$ is bounded it would be enough to consider the chromatic number of the family $\{\widehat{G}^{-} \mid \widehat{G} \in \mathcal{G}\}$ where $\widehat{G}^{-}$ is chosen for an arbitrary signature among all signatures equivalent to the signature in $\widehat{G}$. 
	
	The study can be compared to the notion of GS sets for dichromatic number of digraphs introduced in \cite{ACN21}. For more recent study on this subject, see \cite{AACT} and references therein. For dichromatic number the vertices of a digraph are to be partitioned into sets none of which induces a directed cycle. While the two notions of dichromatic number and balanced chromatic number are quite similar, there are two essential differences between the two. The first is that being balanced is more restrictive for a color class in signed graphs than being acyclic is in digraphs. That is because among all orientations of a cycle only two are directed, while among all assignments of signs half are unbalanced. The second reason is that balanced coloring is arguably more suitable than even proper coloring (of graphs) to study the connection between minor theory and coloring. This is one of the main motivations for studying the balanced chromatic number. For further comments on this connection see \cite{JMNNQ25}. 
	
	\section{Signed \Gya-Sumner sets}
	
	Given a graph $G$, the signed graph obtained from $G$ by replacing each edge with a digon is denoted by $\widetilde{G}$.
	Independent sets in $G$ correspond to balanced sets in $\widetilde{G}$. Hence, $\chi_{b}(\widetilde{G})=\chi(G)$. Furthermore, observe that $\set{\widetilde{G}\mid G \text{ is a graph}} = \F {\left\{\widehat{(K_2,-)}\right\}}$.
	Thus, the \Gya-Sumner conjecture can be restated as follows. 
	
	\begin{conjecture} 
		For any forest $F$ and any complete graph $K_t$ the set $X=\{\widehat{(K_2,-)}, \widetilde{K}_t,\widetilde{F} \}$ is a $GS$ set.
	\end{conjecture}
	
	In other words, the claim is that any minimal $GS$ set containing $\widehat{(K_2, -)}$ is of order three except for the trivial cases: $\mathcal{F}=\{\widehat{(K_2,-)}, \widetilde{K}_2 \}$ and $\mathcal{F}=\{K_1 \}$. 
	Our question can be restated as finding minimal $GS$ sets including $\widetilde{K}_2$, but for simplicity we will not repeat this element, and only consider simple graphs. Toward characterizing such sets, the following is a key observation.  
	
	\begin{observation}
		If $\mathcal{G}$ is a hereditary class of signed graphs such that $\chi_b(\mathcal{G})$ is unbounded, then any $GS$ set must contain a signed graph that is switching equivalent to a member of $\mathcal{G}$.
	\end{observation}
	
	Basic families of hereditary signed simple graphs with unbounded balanced chromatic number are the followings.
	
	\begin{itemize}
		\item $\mathcal{K}= \set{(K_i, -) \mid i\geq 1}$.
		\item $\mathcal{G}_{k}= \set{(G, -) \mid \text{$G$ has girth at least $k$}}$. 
		\item  $PC(\mathcal{G}_{k})= \set{ (K_{|V(G)|}, E(G)) \mid \text{$G$ has girth at least $k$}}$. 
	\end{itemize}  
	The first two classes are cliques and graphs of high girth. The third class is obtained from the second by replacing all non-edges with positive edges. We call this operation the \textit{positive completion} (denoted $PC(\bullet)$). $PC(C_5)$ is drawn in \Cref{fig:forb} 
	
	\begin{lemma}
		The classes of signed graphs $\mathcal{K}$, $\mathcal{G}_{k}$ and $PC(\mathcal{G}_k)$ have unbounded balanced chromatic number.
	\end{lemma}
	
	\begin{proof}
		In $\mathcal{K}$, maximal balanced sets are of size 2, therefore $\chi_b(K_i,-) \geq \lceil \frac{i}{2} \rceil$ and the class has unbounded balanced chromatic number. 
		
		The class $\mathcal{G}_{k}$ and $PC(\mathcal{G}_k)$ have the same negative edges : graphs of girth at least $k$, which have unbounded chromatic number from  \Cref{thm:large_girth}. Hence, by \Cref{lem:no_switch}, those two classes also have unbounded balanced chromatic number. 
	\end{proof}
	An immediate corollary is the following :
	
	\begin{corollary}\label{cor:GS_form}
		If $\mF$ is a GS set, then there exists an integer $n$ and two forests $F_1,F_2$, such that 
		
		\centering $\set{\widehat{(K_t,-)},\widehat{PC(F_1)},F_2} \subset \mF$
		
	\end{corollary}

	In the case that $t=2$ or one of $F_1$ or $F_2$ has two vertices, the family $\mF$ consists of empty graphs. Otherwise, if $t\in \set{3,4}$, then $\widehat{(K_t,-)}$ and $\widehat{PC(F_1)}$ could be switch equivalent, in which case we can have a GS set of order 2. This work is then a first step toward characterizing GS sets of order 2. 
	
	Our contribution is depicted in \Cref{tab:result}, where a linear forest is a forest whose components are paths.
	
	\begin{table}[h]
		\centering
		\renewcommand{\arraystretch}{1.5}
\begin{tabular}{c|c|c|}
\cline{2-3}
 & $K_n= K_3$ & $K_n= K_4$ \\ \hline
\multicolumn{1}{|c|}{$F$ contains a vertex of degree $3$  }& \multicolumn{2}{|c|}{UNBOUNDED} \\ \hline
\multicolumn{1}{|c|}{$F$ is a linear forest and $P_5 \not \subseteq F $ }& \multirow{2}{*}{BOUNDED} & BOUNDED \\ \cline{1-1} \cline{3-3}
\multicolumn{1}{|c|}{$F$ is a linear forest and $P_5 \subseteq F$ }& & UNKNOWN \\ \hline
\end{tabular}  
		\caption{Results on $\chi_b\left(\F{\left\{\widehat{(K_n,-)},F\right\}}\right)$}
		\label{tab:result}
	\end{table}

	\section{Sequences of unbounded balanced chromatic number}
	
	In this section we present some sequences of graphs with unbounded balanced chromatic number excluding some particular red-blue induced subgraph.
	
	Our first construction is based on the notion of shift graphs studied in \cite{EH64}. The second construction is based on the notions of line graphs and arc graphs. Ultimately, the line graph construction will be stronger than the shift graph one, but since the shift graph construction is self contained and more explicit, we believe that both of them might be interesting to the reader.
	
	\subsection{Signed shift graphs} 
	
	Given positive integers $k$ and $n$, $k \leq n$, the shift graph $S_{k,n}$ has as its vertices all increasing sequences $(s_1, s_2,\ldots, s_k)$, $1\leq s_1\leq s_2 \leq \cdots \leq s_k\leq n$, where two sequences are adjacent if they are of the form  $(s_1, s_2,\ldots, s_k)$ and $(s_2, s_3,\ldots, s_{k+1})$.

	Given an integer $k$, the family of shift graphs $\{S_{k,n} \mid n \in \mathbb{N}  \}$, is an example of a family of triangle-free graphs of unbounded chromatic number. For the sake of completeness, we present a proof of this fact first proved in \cite{EH68}.

	\begin{lemma}\label{lem:chi-s(kn)}
		Given any positive integers $k$ and $c$, there exists an integer $n$ such that $\chi (S_{k,n})>c$.
	\end{lemma}

	\begin{proof}
		We prove this by induction on $k$. For $k=1$, $S_{1, n}=K_n$ and the claim is immediate.  The claim follows if we show that $\chi(S_{k,n}) \leq 2^{\chi(S_{k+1, n+1})}$.
		
		Let $\phi$ be a $c$-coloring of $S_{k+1,n+1}$. For each vertex $s=(s_1,s_2, \ldots, s_{k})$ of $S_{k,n}$ let $\psi(s)$ be the set of all colors assigned to the vertices $(s_1,s_2, \ldots, s_{k}, t)$ of $S_{k+1,n+1}$ such that $s_k < t \leq n+1$. We claim that $\psi$ is a (proper) coloring of $S_{k,n}$. That is because given a pair $u'=(u_1,u_2, \ldots u_{k})$ and $u''=(u_2,u_3, \ldots u_{k}, u_{k+1})$ of adjacent vertices in $S_{k, n}$, we consider the vertex  $u=(u_1,u_2, \ldots u_{k}, u_{k+1})$ of $S_{k+1,n+1}$ and observe that $\phi(u) \in \psi(u')$, but $\phi(u)\notin \psi(u'')$ because every extension of $u''$ is adjacent to $u$. Thus $\psi(u')\neq \psi(u'')$. 
	\end{proof}

	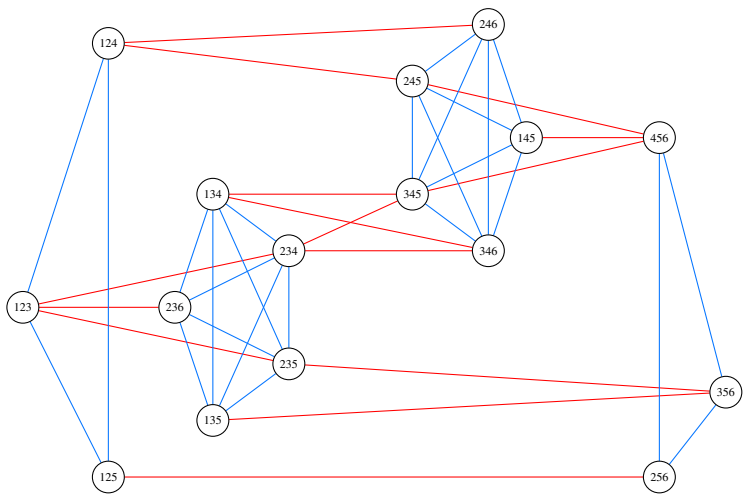
\begin{figure}[h]
		\centering
		\begin{tikzpicture}[scale = .5]
	\begin{pgfonlayer}{nodelayer}
		\node [style=small white] (0) at (-12.25, 2.5) {\small 123};
		\node [style=small white] (1) at (-10, 9.5) {\small 124};
		\node [style=small white] (2) at (-10, -2) {\small 125};
		\node [style=small white] (3) at (-5.25, 1) {\small 235};
		\node [style=small white] (4) at (-8.25, 2.5) {\small 236};
		\node [style=small white] (5) at (-7.25, 5.5) {\small 134};
		\node [style=small white] (6) at (-7.25, -0.5) {\small 135};
		\node [style=small white] (8) at (-5.25, 4) {\small 234};
		\node [style=small white] (9) at (1, 7) {\small 145};
		\node [style=small white] (11) at (-2, 8.5) {\small 245};
		\node [style=small white] (12) at (0, 10) {\small 246};
		\node [style=small white] (13) at (-2, 5.5) {\small 345};
		\node [style=small white] (14) at (0, 4) {\small 346};
		\node [style=small white] (15) at (4.5, 7) {\small 456};
		\node [style=small white] (16) at (4.5, -2) {\small 256};
		\node [style=small white] (17) at (6.25, 0.25) {\small 356};

	\end{pgfonlayer}
	\begin{pgfonlayer}{edgelayer}
		\draw [style=edge red] (0) to (8);
		\draw [style=edge red] (0) to (3);
		\draw [style=edge red] (0) to (4);
		\draw [style=edge red] (1) to (11);
		\draw [style=edge red] (1) to (12);
		\draw [style=edge red] (2) to (16);
		\draw [style=edge red] (8) to (13);
		\draw [style=edge red] (8) to (14);
		\draw [style=edge red] (3) to (17);
		\draw [style=edge red] (5) to (13);
		\draw [style=edge red] (5) to (14);
		\draw [style=edge red] (6) to (17);
		\draw [style=edge red] (9) to (15);
		\draw [style=edge red] (11) to (15);
		\draw [style=edge red] (13) to (15);
		\draw [style={blue_edge}] (17) to (16);
		\draw [style={blue_edge}] (16) to (15);
		\draw [style={blue_edge}] (15) to (17);
		\draw [style={blue_edge}] (1) to (0);
		\draw [style={blue_edge}] (0) to (2);
		\draw [style={blue_edge}] (2) to (1);
		\draw [style={blue_edge}] (5) to (8);
		\draw [style={blue_edge}] (8) to (4);
		\draw [style={blue_edge}] (4) to (3);
		\draw
		[style={blue_edge}] (4) to (5);
		\draw
		[style={blue_edge}] (4) to (6);
		\draw [style={blue_edge}] (3) to (6);
		\draw [{blue_edge}] (3) to (8);
		\draw [style={blue_edge}] (6) to (5);
		\draw [style={blue_edge}] (5) to (3);
		\draw [style={blue_edge}] (8) to (6);
		\draw [style={blue_edge}] (12) to (11);
		\draw [style={blue_edge}] (11) to (9);
		\draw [style={blue_edge}] (9) to (12);
		\draw [style={blue_edge}] (9) to (13);
		\draw [style={blue_edge}] (11) to (13);
		\draw [style={blue_edge}] (9) to (14);
		\draw [style={blue_edge}] (13) to (14);
		\draw [style={blue_edge}] (14) to (12);
		\draw [style={blue_edge}] (12) to (13);
		\draw [style={blue_edge}] (14) to (11);
	\end{pgfonlayer}
\end{tikzpicture}
		\caption{The biggest connected component of the signed shift graph $\widehat{S}_{3,6}$}
		\label{fig:shift_graph}
	\end{figure}
	
	We now consider the signed graph $\widehat{S}_{3,n}$ built on $S_{3,n}$ as follows. The negative edges of $\widehat{S}_{3,n}$ are the edges of  $S_{3,n}$. Each pair of vertices, $(a, b, c)$ and $(a', b, c')$, is connected by a positive edge. The main connected component of $\widehat{S}_{3,6}$ is depicted in \Cref{fig:shift_graph}.  Using the signed graph $\widehat{S}_{3,n}$ we can prove the following:   
	
	\begin{theorem}\label{thm:K3+K1,4}
		The set $\set{\widehat{(K_3,-)},K_{1,4}}$ is not a GS set.
	\end{theorem}
	
	\begin{proof}
		By \Cref{cor:link_signed_unsigned} and \Cref{lem:chi-s(kn)}, the class of signed graphs $\widehat{S}_{3,n}$ has unbounded balanced chromatic number. 
		It remains to prove that each signed graph $\widehat{S}_{3,n}$ is in $\F{\left\{\widehat{(K_3,-)},K_{1,4}\right\}}$.
		Observe that the set of positive edges induces a disjoint union of cliques (each corresponding to a middle value of the triplets). Furthermore, the negative neighbors of a vertex $(a,b,c)$ are partitioned into two types: $(\bullet, a, b)$ and $(b,c, \bullet)$, each of which induces a clique where all the edges are positive. Thus, there can only be a maximum of three independent neighbors of a vertex $(a,b,c)$: a positive neighbor $(\bullet, b, \bullet)$, two negative neighbors $(\bullet, a, b)$ and $(b,c, \bullet)$. Furthermore, since $S_{3,n}$ is triangle-free and since positive edges induce unions  of cliques, $\widehat{S}_{3,n}$ has no negative triangle. 
	\end{proof}
	

	\subsection{Signed line graphs}
	
	The line graph of a graph $G$, denoted $L(G)$, has as vertices the edges of $G$, where edges with common end point in $G$ are adjacent in $L(G)$. Various characterizations of line graphs are given in the literature, among which is Beineke's characterization \cite{Beineke70}: a graph $H$ is isomorphic to the line graph of a simple graph $G$ if and only if it has no induced subgraph isomorphic to one of the nine graphs known as Beineke's. The first of those is the claw ($K_{1,3}$), which means, in particular, that line graphs are claw free.
	
	Given an orientation $D$ of $G$, the arc graph of $D$, denoted $A(D)$, is the graph whose vertices are arcs of $D$ and where the arcs $uv$ and $vw$ are adjacent. A relation between the chromatic number of $A(D)$ and the chromatic number of the graph $G$ was given in \cite{HE72, Poljak91}.
	
	\begin{theorem}\label{thm:ArcChromaticNumber}
		Given a graph $G$ and an orientation $D$ of it, we have 
		
		$$\min \set{k \mid \chi(G)\leq 2^k} \leq \chi(A(D)) \leq \min \set{k \mid \chi(G)\leq \binom{k}{\lfloor \frac{k}{2} \rfloor}}.$$	
		
	\end{theorem}

	\begin{definition}
		Given a graph $G$ and an orientation $D$ of it, we define the signed line graph of $D$, denoted $(L(G), \sigma_D)$ to be a signed graph on $L(G)$ whose negative edges are the edges of $A(D)$.  
	\end{definition} 	
	
	Observe that a triangle in $L(G)$ corresponds to one of two possibilities: either a $K_3$ in $G$, or a $K_{1,3}$ in $G$. Any triangle in $L(G)$ which corresponds to a $K_{1,3}$ in $G$ is of positive sign in $(L(G), \sigma_D)$. That is because either the three edges of $K_{1,3}$ are all oriented the same way, in which case they induce three positive edges in $(L(G), \sigma_D)$, or exactly two of them are in the same direction, in which case we have a triangle with exactly two negative edges in $(L(G), \sigma_D)$.
	
	If the graph $G$ is selected to be triangle-free, then all of the triangles of $(L(G), \sigma_D)$ are of the second type and. Hence, every triangle is positive.
	
	By taking $G$ to be a triangle-free graph of arbitrarily large chromatic number (for example using \Cref{thm:large_girth}), then applying \Cref{thm:ArcChromaticNumber}, we conclude that the negative edges of $(L(G), \sigma_D)$ induce a graph of high chromatic number. Hence, from \Cref{lem:no_switch}, we conclude that the balanced chromatic number of $(L(G), \sigma_D)$ can be arbitrarily large.
	
	Noting that line graphs are in particular claw free, we conclude that:
	
	\begin{theorem}
		\label{thm:K_3+K_1,3}
		The set $\set{\widehat{(K_3,-)},K_{1,3}}$ is not a GS set.
	\end{theorem}

	\section{Negative triangle and linear forest}
	
	So far we have observed that, for a set $\mF=\set{\widehat{F_1}, \widehat{F_2}}$ to be a GS set, one of  $\widehat{F_1}, \widehat{F_2}$, say $\widehat{F_1}$, must be switching equivalent to either $(K_3,-)$ or $(K_4,-)$ and the other must be a forest. If the forest $F_2$ has a vertex of degree $3$ or more, then $\F{(\mF)}$ contains  $\F{\{\widehat{(K_3,-)},K_{1,3}\}}$ and, therefore, its balanced chromatic number is not bounded. Thus, for $\mF$ to be a GS set of order two, besides the fact that the underlying graph of $F_1$ must be $K_3$ or $K_4$, the second, $F_2$, must be a linear forest: that is a forest where each component is a path. In this section, then, we show that for each linear forest $F$,  $\mF=\set{\widehat{(K_3,-)}, F}$ is a GS set. To that end, we first show that it is enough to only consider the cases when $F$ is a path, i.e., a connected linear forest. The disjoint union of two graphs $G$ and $H$ is denoted $G+H$.

	\begin{proposition}\label{prop:K3-ForestToPath}
		If $\set{\widehat{(K_3,-)}, F_1}$ and $\set{\widehat{(K_3,-)}, F_2}$ are GS sets, then $\set{(K_3,-), F_1+F_2}$ is also a GS set.
	\end{proposition}
	
	\begin{proof}
		Suppose $\chi_b\left(\F{\set{\widehat{(K_3,-)}, F_1}}\right)\leq s$ and $\chi_b\left(\F{\set{\widehat{(K_3,-)}, F_2}}\right)\leq t$. We now consider a signed graph $\widehat{G}\in \F{\set{\widehat{(K_3,-)}, F_1+F_2}}$ and claim that $\chi_b(\widehat{G}) \leq \max\{s, |F_1|+t\}$. Assume to the contrary, that $\chi_b(\widehat{G}) \geq \max\{s+1, |F_1|+t+1\}$.
		
		\noindent
		\begin{minipage}{0.71\linewidth}
			\setlength{\parindent}{2em}
			Observe that being a $\widehat{(K_{3},-)}$-free signed graph is the same as the (closed) neighborhood of each vertex inducing a balanced signed graph. That is because given a vertex $v$, after switching at some neighbors if needed, we may have only positive edges incident to $v$. Then there is no negative triangle incident with $v$ if and only if there is no negative edge induced by its neighbors. 
			
			As $\chi_b(\widehat{G})> s$, and since there is no $\widehat{(K_3,-)}$, there must be an induced copy $F'$ of $F_1$. Consider the subgraph $\widehat{G}_1$ induced by $F'$ and all its neighbors. We claim that $\chi_b(\widehat{G}_1) \leq |F_1|$. Indeed, $V(G_1) = \bigcup\limits_{u \in F'} N[u]$, and each of these $ |F_1|$ sets are balanced. 
			
			As $\chi_b(\widehat{G})\geq |F_1|+t+1$, we have $\chi_b(\widehat{G}-\widehat{G}_1)\geq t+1$. As $\widehat{G}-\widehat{G}_1$ still has no $\widehat{(K_3,-)}$, it must have an induced subgraph $F''$ isomorphic to $F_2$. See \Cref{fig:disjoint_union_proof} for an illustration. Then, $F'$ and $F''$ together induce an isomorphic copy of $F_1+F_2$, contradicting the choice of $\widehat{G}$. \qedhere			
		\end{minipage}\hfill
		\begin{minipage}{0.22\linewidth}
			\centering
		
			\begin{tikzpicture}[scale = .5]
	\begin{pgfonlayer}{nodelayer}
		\node [style=none] (16) at (-4.25, 7.25) {};
		\node [style=none] (17) at (-3.25, 7.25) {};
		\node [style=none] (18) at (-2.25, 6.25) {};
		\node [style=none] (19) at (-2.25, -1) {};
		\node [style=none] (20) at (-3.25, -2) {};
		\node [style=none] (21) at (-4.25, -2) {};
		\node [style=none] (22) at (-5.25, -1) {};
		\node [style=none] (23) at (-5.25, 6.25) {};
		\node [style=small white] (24) at (-3, 5) {};
		\node [style=small white] (25) at (-3, 3) {};
		\node [style=small white] (26) at (-3, 1) {};
		\node [style=small white] (27) at (-3, -1) {};
		\node [style=small white] (28) at (-3.75, 6.25) {};
		\node [style=small white] (29) at (-4.5, 6.5) {};
		\node [style=small white] (30) at (-4.25, 5.75) {};
		\node [style=small white] (31) at (-3.75, 4.25) {};
		\node [style=small white] (32) at (-4.5, 4.5) {};
		\node [style=small white] (33) at (-4.25, 3.75) {};
		\node [style=small white] (34) at (-3.75, 2.25) {};
		\node [style=small white] (35) at (-4.5, 2.5) {};
		\node [style=small white] (36) at (-4.25, 1.75) {};
		\node [style=small white] (37) at (-3.75, 0.25) {};
		\node [style=small white] (38) at (-4.5, 0.5) {};
		\node [style=small white] (39) at (-4.25, -0.25) {};
		\node [style=small white] (40) at (-1, 5) {};
		\node [style=small white] (41) at (-1, 3) {};
		\node [style=small white] (42) at (-1, 1) {};
		\node [style=small white] (43) at (-1, -1) {};
		\node [style=none] (44) at (-3.75, -2.5) {$\widehat{G}_1$};
		\node [style=none] (45) at (-3, -1.5) {$F'$};
		\node [style=none] (46) at (-1, -1.5) {$F''$};
	\end{pgfonlayer}
	\begin{pgfonlayer}{edgelayer}
		\draw [style={gray_region}] (20.center)
			 to (21.center)
			 to [bend left=45] (22.center)
			 to (23.center)
			 to [bend left=45, looseness=1.25] (16.center)
			 to (17.center)
			 to [bend left=45] (18.center)
			 to (19.center)
			 to [bend left=45] cycle;
		\draw [style=edge red] (24) to (25);
		\draw [style=edge red] (25) to (26);
		\draw [style=edge red] (26) to (27);
		\draw [style={blue_edge}] (24) to (30);
		\draw [style={blue_edge}] (29) to (24);
		\draw [style={blue_edge}] (28) to (24);
		\draw [style={blue_edge}] (28) to (29);
		\draw [style={blue_edge}] (29) to (30);
		\draw [style={blue_edge}] (30) to (28);
		\draw [style={blue_edge}] (31) to (32);
		\draw [style={blue_edge}] (32) to (33);
		\draw [style={blue_edge}] (33) to (31);
		\draw [style={blue_edge}] (34) to (35);
		\draw [style={blue_edge}] (35) to (36);
		\draw [style={blue_edge}] (36) to (34);
		\draw [style={blue_edge}] (37) to (38);
		\draw [style={blue_edge}] (38) to (39);
		\draw [style={blue_edge}] (39) to (37);
		\draw [style={blue_edge}] (33) to (25);
		\draw [style={blue_edge}] (25) to (31);
		\draw [style={blue_edge}] (32) to (25);
		\draw [style={blue_edge}] (34) to (26);
		\draw [style={blue_edge}] (26) to (36);
		\draw [style={blue_edge}] (35) to (26);
		\draw [style={blue_edge}] (37) to (27);
		\draw [style={blue_edge}] (27) to (39);
		\draw [style={blue_edge}] (38) to (27);
		\draw [style=edge red] (40) to (41);
		\draw [style=edge red] (41) to (42);
		\draw [style=edge red] (42) to (43);
	\end{pgfonlayer}
\end{tikzpicture}
			
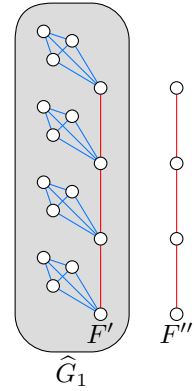
\captionof{figure}{Illustration of the proof, with $F_1=F_2=P_4$}
			\label{fig:disjoint_union_proof}
		\end{minipage}
	\end{proof}

	\begin{theorem}\label{thm:K3+Pk}
		For every positive integer $k$, the set $\set{\widehat{(K_3,-)}, P_k}$ is a GS set.
	\end{theorem}
	
	The following proof is inspired from the classic argument from \Gya \cite{Gya87}.
	\begin{proof}
		We prove, by induction on $k$, that $\chi_{b}\left(\F{\set{\widehat{(K_3,-)}, P_{k+2}}}\right)< 2^k.$ 		
		In fact, we prove a stronger claim: given a connected signed graph $\widehat{G}$ in $\F{\set{\widehat{(K_3,-)}}}$ whose balanced chromatic number is at least $2^{k}$, for each vertex $u$ of $G$ there is an induced path of length $k+1$ (i.e. $P_{k+2}$) starting at $u$.
		
		For $k=1$, if there is no path of length 2 starting at the vertex $u$, noting that $G$ is connected, we conclude that $\widehat{G}$ is a switching of $(K_i, +)$ for some $i$. Thus, $\chi_b(\widehat{G})=1$. 
		
		Suppose that the claim holds for every value up to $k-1$ and consider the statement for $k$.  Let $\widehat{G}$  be a connected signed graph in $\F{\set{\widehat{(K_3,-)}}}$ which has balanced chromatic number at least $2^{k}$ and let $u$ be a vertex of $\widehat{G}$. Let $\widehat{G}_d$ be the subgraph of $\widehat{G}$ induced by the vertices at distance $d$ from $u$. Assume, toward a contradiction, that there is no induced path of length $k+1$ starting from $u$. That is to say: $\set{u}, \widehat{G}_1, \widehat{G}_2, \ldots, \widehat{G}_k$ cover all of the vertices of $\widehat{G}$. 
		
		We claim that, for some $i$, we have $\chi_b(\widehat{G}_i)\geq 2^{k-i}+1$. Observe that $u$ and $\widehat{G}_1$ together induce a balanced set. If each $\widehat{G}_i$ is $2^{k-1}$-colorable for $i=2, \dots, k$, then $\chi_b(\widehat{G})\leq 1+2^{k-1}+2^{k-2} \ldots +2=2^{k}-1$, contradicting the assumption that $\chi_b(\widehat{G}) \geq 2^k$. Thus, we assume that $\chi_b(\widehat{G}_i)\geq 2^{k-i}+1$ for a fixed $i$, with $2\leq i \leq k$.
		
		In $\widehat{G}_{i}$, one of the connected components, say $\widehat{G}_i^1$ has balanced chromatic number at least $2^{k-i}+1$. Let $u_{i-1}$ be a vertex in $\widehat{G}_{i-1}$ that connects $\widehat{G}_{i}^1$ to $u$. Since the neighborhood of each vertex has balanced chromatic number $1$. The signed graph $\widehat{G}_{i}^1-N(u_{i-1})$ has balanced chromatic number at least $2^{k-i}$. Thus, one of its connected components, say $\widehat{C}$, has balanced chromatic number at least $2^{k-i}$. Let $u_{i}$ be a vertex in $N(u_{i-1})$ which has a neighbor in $\widehat{C}$. The subgraph $\widehat{C'}$ induced by $\widehat{C}$ and $u_i$ is connected, belongs to $\F{\set{\widehat{(K_3,-)}}}$, and has balanced chromatic number at least $2^{k-i}$. Thus, by the inductive assumption, we have an induced path $P$ of length $k-i+1$ in $\widehat{C'}$ starting at $u_i$.
		By the choice of $\widehat{C'}$, $u_i$ is the only neighbor of $u_{i-1}$ in $P$. Hence, extending $P$ to $u_{i-1}$ and then, through a shortest path, to $u$, we have an induced path of length $k+1$ with $u$ as its starting point.
	\end{proof}


	\begin{corollary}
		If $F$ is linear forest consisting of $l$ paths each of length at most $k$, then  $$\chi_{b}\left(\F{\set{\widehat{(K_3,-)}, F}}\right)< 2^k + (l-1)k.$$
	\end{corollary}

	\section{$\widehat{(K_4,-)}$ and linear forest}

	We have already seen that, for $\set{\widehat{(K_4,-)}, F}$ to be a GS set, $F$ must be a linear forest. We conjecture that this necessary condition is also sufficient.
	
	\begin{conjecture}\label{conj:K4+BambooForest}
		For any linear forest $F$, the set $\set{\widehat{(K_4,-)}, F}$ is a GS set.
	\end{conjecture} 
	
	As an approach to this conjecture, we first show that it would be enough to prove the conjecture for when $F$ is just a path. Then we verify it for $F=P_4$, concluding the conjecture for when each component of $F$ is a path of length at most 3. The first step is an extension of \Cref{thm:K3+Pk} for signed graphs with no induced $\widehat{(K_4,-)}$. Observe that a key element in the proof of  \Cref{thm:K3+Pk} is the fact that the (closed) neighborhood of each vertex is balanced, that is a consequence of the assumption that there is no induced $\widehat{(K_3,-)}$. When $\widehat{(K_4,-)}$ is forbidden instead, we do not know if the neighborhood of a vertex has bounded balanced chromatic number (see \Cref{conj:K4FreeNieghborhood}). Thus, in order to extend \Cref{thm:K3+Pk}, we add this as an assumption.
	
	\begin{theorem}\label{thm:NeighborhoodImpliesUniversalBound}
		Given two positive integers $b$ and $k$, with $k\geq 3$, let $\widehat{G} \in \F{\set{\widehat{(K_4,-)}, P_k}}$ be a signed graph with the property that the closed neighborhood of each vertex admits a balanced $b$-coloring. Then $\chi_b(\widehat{G})\leq b2^{k-3}$.
	\end{theorem}
	
	\begin{proof}
		We prove the claim by induction on $k$ for the family of all such graphs. More precisely we claim the following. Let $\widehat{G}$ be a connected signed graph with no induced $\widehat{(K_4,-)}$ whose balanced chromatic number is larger than $b2^{k-3}$, but with the property that the closed neighborhood of each vertex admits a balanced $b$-coloring. Then for each vertex $u$ of $\widehat{G}$ there is an induced path of length at least $k-1$ starting at $u$. 
		
		For $k=3$, if there is no $P_3$ starting at $u$, then $N[u]$ spans the whole graph and, therefore, by assumption, $\chi_b(\widehat{G})\leq b$. Given $k\geq 4$, assume that the claim holds for every value up to $k-1$ and consider a signed graph $\widehat{G}$ satisfying the condition for $k$. Suppose that there is a vertex $u$ of $\widehat{G}$ not satisfying the conclusion. That is to say: $\set{u}, \widehat{G}_1, \widehat{G}_2, \ldots, \widehat{G}_{k-3}$ cover all of the vertices of $\widehat{G}$, where $\widehat{G}_i$ is the subgraph of $\widehat{G}$ induced by the vertices at distance $i$ from $u$. Since $\chi_b(\widehat{G}) > b2^{k-3}$, one of these subgraphs, say $\widehat{G}_{i}$, has balanced chromatic number larger than $b2^{k-i-3}+b$.
		Then, as in the proof of \Cref{thm:K3+Pk}, by taking a suitable component of $\widehat{G}_{i}$, selecting a vertex $u_{i}$ with a neighbor $u_{i-1}$ in $\widehat{G}_{i-1}$, deleting all but one of the neighbors of $u_{i-1}$ in $\widehat{G}_{i}$, then applying induction, we get an induced path of length $k-i-1$ in $\widehat{G}_{i}$ with only one connection to $u_{i-1}$. This connection then can be extended with a shortest path connecting $u_{i-1}$ to $u$ to produce an induced path of length $k-1$.  
	\end{proof}

	In light of this development, the following relaxation of \Cref{conj:K4+BambooForest} seems to be essential for resolving the conjecture. 
	
	\begin{conjecture}\label{conj:K4FreeNieghborhood}
		For any positive integer $k$, there exists a positive integer $b_k$ such that for any signed graph $\widehat{G}\in \F{\set{\widehat{(K_4,-)}, P_k}}$ and any vertex $u$ of it, the closed neighborhood of $u$ has balanced chromatic number at most $b_k$.  
	\end{conjecture} 
	
	In fact the two conjectures are equivalent: If \Cref{conj:K4+BambooForest} holds, then we can take the bound on the balanced chromatic number of $\F{\set{\widehat{(K_4,-)}, P_k}}$ to be $b_k$. Conversely, if \Cref{conj:K4FreeNieghborhood} holds, then we apply \Cref{thm:NeighborhoodImpliesUniversalBound} to get a bound on the balanced chromatic number of $\F{\set{\widehat{(K_4,-)}, P_k}}$.
	
	Therefore, we aim at understanding better the conditions on the neighborhood. To that end, considering a vertex $u$ of $\widehat{G}\in \F{\set{\widehat{(K_4,-)}, P_k}}$, we first switch at each neighbor of $u$ which is adjacent to it by a negative edge. That renders us with a vertex $u$ all whose neighbors are connected to it by positive edges. The fact that $\widehat{G}\in \F{\set{\widehat{(K_4,-)}, P_k}}$ has two implications. The first is that the open neighborhood $\widehat{G}_u$ of the vertex $u$ is $P_k$-free, that is because it is an induced subgraph of $\widehat{G}$. The second is that $\widehat{G}_u$ has neither $(K_3,-)$ nor $(K_4, M)$ as induced subgraph. Here $(K_4, M)$ is the signed graph on $K_4$ with a matching of size two being the negative edges (see \Cref{fig:forb_graphs}). We note that these (induced) subgraphs are forbidden only with the given signature. The first is forbidden, as otherwise, together with $u$, it would form a subgraph switching equivalent to $(K_4,-)$. The second is itself switching equivalent to $(K_4,-)$. Moreover, this is the only switching equivalent copy of $(K_4, -)$ which has no $(K_3,-)$ subgraph. Overall \Cref{conj:K4FreeNieghborhood} can be reformulated as follows.

	\begin{conjecture}\label{conj:K4FreeNieghborhood-2}
		For any positive integer $k$, there exists a positive integer $b_k$ such that any signed graph in $\F{\{(K_3,-), (K_4,M), P_k\}}$ has balanced chromatic number at most $b_k$.  
	\end{conjecture}  
	
	\begin{figure}[h]
		\centering
		\begin{tikzpicture}[scale = .5]
	\begin{pgfonlayer}{nodelayer}
		\node [style=small white] (0) at (0, 1) {};
		\node [style=small white] (1) at (-1.25, -1) {};
		\node [style=small white] (2) at (1.25, -1) {};
		\node [style=small white] (3) at (3, 1) {};
		\node [style=small white] (4) at (3, -1) {};
		\node [style=small white] (5) at (5, -1) {};
		\node [style=small white] (6) at (5, 1) {};
	\end{pgfonlayer}
	\begin{pgfonlayer}{edgelayer}
		\draw [style=edge red] (0) to (1);
		\draw [style=edge red] (1) to (2);
		\draw [style=edge red] (2) to (0);
		\draw [style=edge red] (3) to (4);
		\draw [style=edge red] (6) to (5);
		\draw [style={blue_edge}] (3) to (6);
		\draw [style={blue_edge}] (6) to (4);
		\draw [style={blue_edge}] (4) to (5);
		\draw [style={blue_edge}] (5) to (3);
	\end{pgfonlayer}
\end{tikzpicture}
		\caption{$(K_3,-)$ and $(K_4,M)$}
		\label{fig:forb_graphs}
	\end{figure}
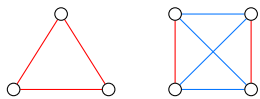

	\begin{proposition}\label{prop:SwitchVSColored}
		For any given integer $k$, \Cref{conj:K4FreeNieghborhood} and \Cref{conj:K4FreeNieghborhood-2} are equivalent.
	\end{proposition}

	\begin{proof}
		Given a signed graph $(G,\sigma)$, let $(G,\sigma)^*$ be the signed graph obtained from $(G, \sigma)$ by adding a universal vertex which is adjacent to all of the vertices with positive edges. The claim of the proposition follows from the fact that  $\widehat{(G, \sigma)}^*\in \F{\set{\widehat{(K_4,-)}, P_k}}$ if and only if $(G, \sigma) \in \F{\{(K_3,-), (K_4,M), P_k\}}$ and noting that the values of $b_k$ in \Cref{conj:K4FreeNieghborhood} will be at most one more than its value in \Cref{conj:K4FreeNieghborhood-2}.
	\end{proof}
	
	Given vertex-disjoint graphs $G_1$ and $G_2$, their \emph{full join}, denoted $G_1 \bowtie G_2$, is a graph on $V(G_1)\cup V(G_2)$ where each of the two sets induces the corresponding graph and each pair of vertices $x\in V(G_1)$ and $y\in V(G_2)$ is adjacent. The notion of full join is crucial in the study of $P_4$-free graphs.
	
	\begin{lemma}\label{lem:FullJoin}
		Assume that $(G_1 \bowtie G_2, \sigma) \in \F{\{(K_3,-), (K_4,M), P_k\}}$. If one of $\widehat{G}_1$ or $\widehat{G}_2$ contains a negative edge, then the subgraph of the other induced by its negative edges has chromatic number at most 3. 
	\end{lemma}

	\begin{proof}
		Assume that $uv$ is a negative edge in $\widehat{G}_1$. Consider the partition of $V(\widehat{G}_2)$ into: 
		$$N_{\widehat{G}_2}^-(u),\ N_{\widehat{G}_2}^-(v),\ N_{\widehat{G}_2}^+(u)\cap N_{\widehat{G}_2}^+(v).$$ 
		Since $(G_1 \bowtie G_2, \sigma)$ has no induced $(K_3,-)$, neither of $N_{\widehat{G}_2}^-(u)$ and $N_{\widehat{G}_2}^-(v)$ induces a negative edge. As $(G_1 \bowtie G_2, \sigma)$ does not induce any $(K_4,M)$, there is no negative edge in $N_{\widehat{G}_2}^+(u)\cap N_{\widehat{G}_2}^+(v)$. Thus, we have a 3-coloring of $\widehat{G}_2^-$ as claimed. See \Cref{fig:FullJoin} for an illustration. 
	\end{proof}
	
	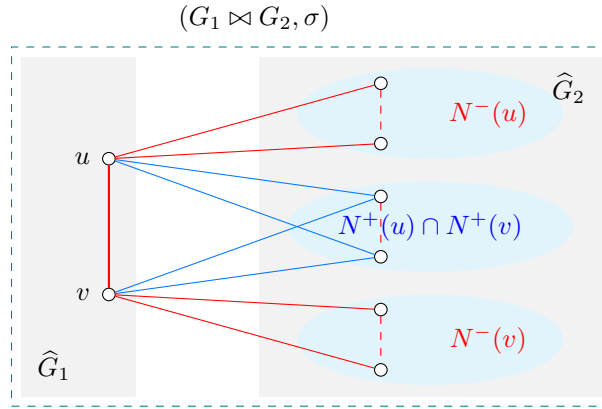
\begin{figure}[h!] 
		\centering 
		\begin{tikzpicture}[scale=0.5, box/.style={draw=none, fill=gray!10, thick}]

    \node[box, minimum width=1.5cm, minimum height=4.5cm] (G1) {};
    \node[box, minimum width=4.5cm, minimum height=4.5cm, right=1.6cm of G1] (G2) {};
    
    \node[anchor=south west, xshift=0.1cm, yshift=0.1cm] at (G1.south west) {$\widehat{G}_1$};
    \node[anchor=north east, xshift=-0.1cm, yshift=-0.1cm] at (G2.north east) {$\widehat{G}_2$};
    
    \node[small white] (u) at ([xshift=0.8cm,yshift=1.8cm]G1) {};
    \node[left=1pt of u] {$u$};
    \node[small white] (v) at ([xshift=0.8cm,yshift=-1.8cm]G1) {};
    \node[left=1pt of v] {$v$};

    \node[draw=teal,dashed,fit=(G1)(G2)] {}; 
    
    \node[above=0.2cm] at ($(G1.north)!0.5!(G2.north)$) {$(G_1\bowtie G_2, \sigma)$};

    \node[draw=none, fill=cyan!10, thick, ellipse,
      minimum width=3.5cm, minimum height=1.2cm,   yshift=1.5cm] (Nu-) at (G2){};
     \node[xshift=-1cm] at (Nu-.east)  {{\textcolor{red}{$N^-(u)$}}};

    \node[draw=none, fill=cyan!10, thick, ellipse,
      minimum width=3.5cm, minimum height=1.2cm,   yshift=0cm] (Nuv+) at (G2){\textcolor{blue}{$N^+(u)\cap N^+(v)$}};

     \node[draw=none, fill=cyan!10, thick, ellipse,
      minimum width=3.5cm, minimum height=1.2cm,   yshift=-1.5cm] (Nv-) at (G2){};
         \node[xshift=-1cm] at (Nv-.east)  {{\textcolor{red}{$N^-(v)$}}};

    \node[small white] (x1) at ([xshift=-1.3cm,yshift=8mm]Nu-) {};
    \node[small white](y1) at ([xshift=-1.3cm,yshift=-8mm]Nu-) {};
    \draw[edge red] (x1)--(u)--(y1);
    \draw[edge red, dashed] (x1)--(y1);

  \node[small white] (x2) at ([xshift=-1.3cm,yshift=8mm]Nuv+) {};
    \node[small white] (y2) at ([xshift=-1.3cm,yshift=-8mm]Nuv+) {};
    \draw[blue_edge] (x2)--(u)--(y2);
    \draw[blue_edge] (x2)--(v)--(y2);
    \draw[edge red, dashed] (x2)--(y2);
    
   \node[small white] (x3) at ([xshift=-1.3cm,yshift=8mm]Nv-) {};
    \node[small white] (y3) at ([xshift=-1.3cm,yshift=-8mm]Nv-) {};
    \draw[edge red] (x3)--(v)--(y3);
    \draw[edge red, dashed] (x3)--(y3);

  \draw[red, thick](u)--(v);
    

\end{tikzpicture} 
		\caption{Illustration of the proof of \Cref{lem:FullJoin}} 
		\label{fig:FullJoin} 
	\end{figure}
	
	\begin{corollary}\label{cor:FullJoin}
		If $(G_1 \bowtie G_2, \sigma) \in \F{\{(K_3,-), (K_4,M), P_k\}}$ and each of $\widehat{G}_1$ and $\widehat{G}_2$ contains a negative edge, then the negative edges induce a $6$-colorable graph, and hence $\chi_b(G_1 \bowtie G_2, \sigma)\leq 6$.	
	\end{corollary}
	
	\begin{proof}
		We have that $\chi((G_1 \bowtie G_2, \sigma)^{-}) \leq \chi(\widehat{G}_1^{-})+\chi(\widehat{G}_2^{-})$.
		By \cref{lem:FullJoin}, the graph induced by the negative edges on each side has chromatic number at most $3$.
	\end{proof}
	
	The following lemma generalises the result of \Cref{prop:K3-ForestToPath} in the perspective of using it with $(K_4,-)$. To better understand its statement, one can think of $\mcal G = \F{\{(K_3,-), (K_4,M)\}}$.

	\begin{lemma}
		\label{lem:magic}
		
		Let $\mcal G$ be a hereditary class of signed graphs, $F_1$ and $F_2$ two forests and $c$ a positive integer such that: Signed graphs in  $\mcal G$ with no induced $F_1$ (respectively $F_2$) have balanced chromatic number at most $c$. Furthermore, assume that for each signed graph $\widehat G \in \mcal G $ and for each induced copy $F_1'$ of $F_1$ in $\widehat G$, the common neighbors of $F_1'$ (with arbitrary signature) induce a $c$-colorable signed graph. Then for some integer $\phi(c)$, any signed graph in $\mcal{G}$ with no induced copy of $F_1+F_2$ is $\phi(c)$-colorable.
	\end{lemma} 
	
	
	\begin{figure}[h]
		\centering
		\begin{tikzpicture}[scale = .5]
	\begin{pgfonlayer}{nodelayer}
		\node [style=small black] (0) at (-8, 0) {};
		\node [style=small black] (1) at (-6, 0) {};
		\node [style=white vertex] (2) at (-4, 0) {};
		\node [style=small black] (3) at (-2, 0) {};
		\node [style=grey vertex] (4) at (-4, 2) {};
		\node [style=none] (5) at (-4, 2) {r};
		\node [style=none] (6) at (-4, 0) {s};
		\node [style=small black] (7) at (2, 0) {};
		\node [style=small black] (8) at (4, 0) {};
		\node [style=small white] (9) at (6, 0.75) {};
		\node [style=small black] (10) at (8, 0) {};
		\node [style=small gray] (11) at (5, 2) {};
		\node [style=small white] (14) at (6, 0) {};
		\node [style=small white] (15) at (6, -0.75) {};
		\node [style=small white] (16) at (6, -1.5) {};
		\node [style=small gray] (17) at (5, 3.5) {};
		\node [style=small white] (18) at (6, -2.25) {};
		\node [style=small gray] (19) at (4, 2) {};
		\node [style=small gray] (20) at (8, 3.5) {};
		\node [style=small gray] (21) at (8, 2) {};
	\end{pgfonlayer}
	\begin{pgfonlayer}{edgelayer}
		\draw [style=Edge] (0) to (1);
		\draw [style=Edge] (1) to (2);
		\draw [style=Edge] (2) to (3);
		\draw [style=Edge] (7) to (8);
		\draw [style=Edge] (8) to (9);
		\draw [style=Edge] (9) to (10);
		\draw [style=Edge] (8) to (14);
		\draw [style=Edge] (14) to (10);
		\draw [style=Edge] (10) to (15);
		\draw [style=Edge] (15) to (8);
		\draw [style=Edge] (8) to (16);
		\draw [style=Edge] (16) to (10);
		\draw [style=Edge] (8) to (18);
		\draw [style=Edge] (18) to (10);
		\draw [style=Edge] (20) to (9);
		\draw [style=Edge] (20) to (14);
		\draw [style=Edge] (20) to (15);
		\draw [style=Edge] (20) to (16);
		\draw [style=Edge] (21) to (16);
		\draw [style=Edge] (15) to (21);
		\draw [style=Edge] (21) to (14);
		\draw [style=Edge] (9) to (21);
		\draw [style=Edge] (11) to (14);
		\draw [style=Edge] (19) to (15);
		\draw [style=Edge] (11) to (9);
		\draw [style=Edge] (9) to (14);
		\draw [style=Edge] (14) to (15);
		\draw [style=Edge] (15) to (16);
	\end{pgfonlayer}
\end{tikzpicture}
		\caption{Illustration of the proof of \Cref{lem:magic} with $F_1=F_2 = P_4$}
		On the left: $F_1^{*}$ , On the right: $\widehat G$ with $Ext_{T \ra S}(T')$ and $Ext_{T \ra R}(T')$ highlighted.\\
		$T$ (resp $T'$) is in black. $S = T\cup \set{s}, R = T\cup \set{r}$
		\label{fig:magic}
	\end{figure}
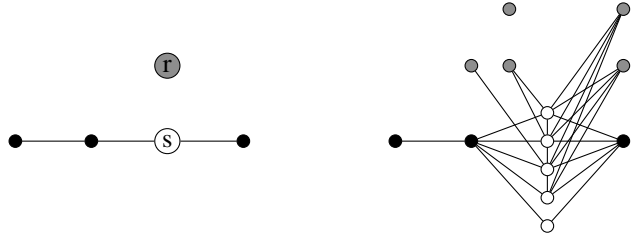

	\begin{proof}
		Throughout the proof $\widehat G$ is an element of $\mcal G \cap \F{\{F_1 + F_2\}}$. Given a subset $X$ of $V(G)$, the subgraph of $\widehat G$ induced by $X$ is denoted by $\widehat G[X]$ (which must also be a member of $\mcal G \cap \F{\{F_1 + F_2\}}$). Note that, in this proof, the graphs we consider are signed. However, since we mostly deal with forests, the signature will not play a big role. We introduce several notations before proceeding with the proof of the lemma.
		\begin{itemize}
			
			\item $F_1^*$ is the forest $F_1 + K_1$ with the added isolated vertex named $h_*$.
			
			\item A $T$-piece in $\widehat G$ is a pair $(T',f )$ where $T'$ is an induced subgraph of $\widehat G$ isomorphic to $T$ and $f$ is an isomorphism mapping $T$ to $T'$.
			
			\item For an induced subforest $T$ of $F_1$, a $1$-extension of $T$ (in $F_1$) is an induced subgraph $S$ of $F_1$, such that $V(S) = V(T) \cup \set{s}$ with $s \notin V(T)$.
			
			\item Given a $T$-piece $(T',f)$ of $\widehat G$ and a $1$-extension $S$ of $T$ in $F_1$, with $V(S) = V(T) \cup \set{s}$, the set $Ext_{T\ra S}(T',f)$, consists of the vertices $s'$ of $\widehat{G}$ such that $(V(T')\cup \set{s'}, f')$ is an $S$-piece, where $f'$ is the extension of $f$ mapping $s$ to $s'$.            
			When it is obvious from the context, we write $Ext_{T\ra S}(T')$.
			Formally, $$Ext_{T\ra S}(T',f) = \bigcap_{\substack {u \in T \\ u \adj[] s}} N(f(u)) \cap  \bigcap_{\substack { v \in T \\ u \not \adj[] s}} \overline{N(f(v))}. \ \ \ (*)$$
			
		\end{itemize}
		See \Cref{fig:magic} for an illustration showcasing two different extenders. 
		
		We prove the following statement by induction: 
		
		\begin{quote}
			\textbf{Claim:} For any $T$-piece $(T',f)$ of $F_1$, there exists a 1-extension $S$ of $T$ in $F_1^{*}$, such that $\chi_b(\widehat G[Ext_{T \ra S}(T')]) \leq c(|F_1|(|F_1|+1))^{2(|F_1|-|T|)}$.
		\end{quote}
		
		The base of induction will be $T=F_1$, then assuming it is true for $T$-pieces on $i$ vertices, we prove that the statement is true for $T$-pieces on $i-1$ vertices. At the end of the inductive process we conclude with $T$ being the empty piece. In this case the set $S$ has a single vertex and the set of $S$-extenders is all the vertices of $\widehat G$. This will give the desired result.
		
		For the base of induction, assume $T=F_1$ and fix an $F_1$-piece $(T',f)$. The only extension of $F_1$ inside $F_1^{*}$ is $F_1^{*}$ itself. The $F_1^{*}$-extenders of $(T',f)$ is the set of vertices disjoint from $F_1'$. Since $\widehat G$ doesn't contain an induced copy of $F_1 + F_2$, this set of extenders doesn't contain and induced copy of $F_2$ and, by the assumption on $F_2$, it has balanced chromatic number bounded by $c = c(|F_1|(|F_1|+1))^{2(|F_1|-|F_1|)}$.
		\vspace{5px}
		
		For the inductive part, assume the claim holds for every subforest of $F_1$ on at least $i$ vertices. Consider an induced subforest $T$ of $F_1$ on $i-1$ vertices and let $(T',f)$ be a $T$-piece in $\widehat G$. Consider a 1-extension $S$ of $T$ inside $F_1^{*}$ with $V(S) = V(T)\cup \set s$ and $s\neq h_*$, noting that we have a choice for $s$ because $T$ has less vertices than $F_1$. If $\chi_b(\widehat G[Ext_{T \ra S}(T')]) \leq c(|F_1|(|F_1|+1))^{2(|F_1|-|T|)}$, then we are done. Hence we assume that $\chi_b(\widehat G[Ext_{T \ra S}(T')]) > c(|F_1|(|F_1|+1))^{2(|F_1|-|T|)}$. 
		
		For any $s' \in Ext_{T \ra S}(T')$, we have an $S$-piece $(S_{s'},g)$ for which the claim applies (by the induction hypothesis).
		We get that there exists an $r_{s'}$ such that the extension $U = F_1^{*}[V(S) \cup \set{r_{s'}}] = F_1^{*}[V(T) \cup\set{s,r_{s'}}]$ verifies $\chi_b(\widehat G[Ext_{S \ra U}(S')]) \leq \ c(|F_1|(|F_1|+1))^{2(|F_1|-|S|)}$. 
		
		Note that there are less than $|F_1|$ 1-extensions of $S$ in $F_1^{*}$. Hence, there is an $r$ such that $$\chi_b(\widehat G[a \in Ext_{T\ra S}(T'), r_a = r]) > c(|F_1|(|F_1|+1))^{2(|F_1|-|T|)}/|F_1|.\ \  (**)$$ Call this subgraph $\widehat A$, fix such an $r$, and let $R = F_1^{*}[V(T) \cup\set r]$. Depending on the adjacency between $s$ and $r$, we consider two cases.
		
		\begin{itemize}
			\item {\bf Case $s \adj[] r$.} Since $\chi(\widehat A) > c$, we can find a copy $F_1'$ of $F_1$ inside $\widehat A$.
			Now consider the $R$-extenders of $T$ in $\widehat G$. $Ext_{T \ra R}(f)$ can be partitioned into: 
			\begin{itemize}
				\item The vertices disjoint from $F_1'$: This set has chromatic number bounded by $c$ because it is $F_2$-free (otherwise it would create a copy of $F_1 + F_2$ inside $\widehat G$).
				\item The vertices adjacent to a given vertex $a \in V(F_1')$: This set correspond exactly to $Ext_{S \ra U} (S',f')$. Where $V(S') = V(T) \cup \set a$ and  $f'(u) = \begin{cases}f(u) & \text{if } u \in V(T) \\  a & \text{if } u = s. \end{cases}$
				Indeed, using \text{formula $(*)$}, we have $Ext_{S \ra U}(S',f') = Ext_{T \ra S}(T',f) \cap N(f'(s)) = Ext_{T \ra S}(T',f) \cap N(a)$.
				
				By definition of $r_a = r$ we get that $\chi(\widehat G[Ext_{S \ra U} (S')]) \leq c(|F_1|(|F_1|+1))^{2(|F_1|-|S|)}$.
			\end{itemize}
			We have successfully partitioned $\widehat A$ into $(|F_1|+1)$ parts each with balanced chromatic number at most $c(|F_1|(|F_1|+1))^{2(|F_1|-|S|)}$. Therefore, $\chi_b(\widehat A) \leq c(|F_1|(|F_1|+1))^{2(|F_1|-|S|)}\times(|F_1|+1) \leq c(|F_1|(|F_1|+1))^{2(|F_1|-|T|)}/|F_1|$, contradicting ($**$).

			\item  {\bf Case $s \not\adj[] r$.} The reasoning is very similar, swapping adjacencies for non-adjacencies.
			
			Since $\chi_b(\widehat A) > c$, we can find a copy $F_1'$ of $F_1$ inside $\widehat A$.
			Now consider the $R$-extenders of $T$ in $\widehat G$. $Ext_{T \ra R}(T')$ can be partitioned into:     	
			\begin{itemize}
				\item The vertices joined to all of $F_1'$ : This set has chromatic number bounded by $c$ by the statement of the lemma.
				\item The vertices non-adjacent to a given vertex $a \in V(F_1')$ : This set corresponds exactly to $Ext_{S \ra U}(S',f')$. Where $V(S') = V(T) \cup \set a$ and $V(S')= V(T')\cup\set{a},\ $ 
				
				and $f'(u) = \begin{cases}f(u) & \text{if } u \in V(T) \\  a & \text{if } u = s. \end{cases}$ By definition of $r_a=r$, we get that this set has balanced chromatic number at most $c(|F_1|(|F_1|+1))^{2(|F_1|-|S|)}$.
			\end{itemize}
			We have successfully partitioned $\widehat A$ into $(|F_1|+1)$ parts each with balanced chromatic number at most $c(|F_1|(|F_1|+1))^{2(|F_1|-|S|)}$. Therefore, $\chi_b(\widehat A) \leq c(|F_1|(|F_1|+1))^{2(|F_1|-|S|)}\times(|F_1|+1) = c(|F_1|(|F_1|+1))^{2(|F_1|-|T|)}/|F_1| $ which contradicts ($**$).
		\end{itemize}			
		
		See \Cref{fig:magic} for an illustration of the proof.
	\end{proof}
	
	\begin{theorem}\label{thm:K4_P4_locally_bounded}
		The class $\F\{(K_3,-), (K_4,M), P_4\}$ has balanced chromatic number at most 6. 
	\end{theorem}

	\begin{proof}
		In fact, we will prove a stronger statement that, for each element $\widehat{G}$ of $\F\{(K_3,-), (K_4,M), P_4\}$, the negative edges of $\widehat{G}$ induce a 6-colorable graph.
		
		Consider a signed graph $\widehat{G}$ in $\F\{(K_3,-), (K_4,M), P_4\}$. As all signatures on $P_4$ are switching equivalent, the underlying graph $G$ of $\widehat{G}$ has no induced $P_4$, in other words $G$ is a cograph. A classic characterization of $P_4$-free graphs is that any $P_4$-free graph $G$, except $K_1$, has nontrivial disjoint subgrpahs $G_1$ and $G_2$ such that either $G=G_1 \bowtie G_2$ or $G=G_1+G_2$. 
		
		We may assume that $G$ is connected, as otherwise we may work with each component separately. Hence, there is partition of $V(G)=A\cup B$ such that $G=G[A]\bowtie G[B]$. If each of $\widehat{G}[A]_-$ and $\widehat{G}[B]_-$ is a 3-colorable subgraph of $\widehat{G}_-$, then $\widehat{G}_-$ is 6-colorable and we are done. Thus, we assume one of them, say $\widehat{G}[B]_-$, is not 3-colorable. Then, by \Cref{lem:FullJoin}, $\widehat{G}[A]$ has no negative edge.

		Let $A'$ be a maximal set of vertices containing $A$ satisfying the following conditions: 
		\begin{enumerate}
			\item It induces no negative edge,
			\item Each connected component of $G\setminus A'$ is a module, that is to say: each vertex outside a component $X$ is either adjacent to all of the vertices of $X$ or to none of them.
		\end{enumerate}  Let $\widehat{H}_{1}, \widehat{H}_{2}, \ldots, \widehat{H}_{\ell}$ be the connected components of $\widehat{G}\setminus A'$. 
		Our goal is tho present a $6$-coloring of $\widehat{G}_-$ where all of the vertices in $A'$ are colored with the same color. To achieve this goal, an $\widehat{H}_i$ satisfying $\chi(\widehat{H}_i) \leq 5$ poses no problem. Thus we assume $\chi(\widehat{H}_{i-})\geq 6$ for each $i$.
		Since each $H_i$ is a connected $P_4$-free graph, it is a full join of two subgraphs, say $H_i'$ and $H_i''$.  
		We first claim that each of $\widehat{H}_{i-}'$ and $\widehat{H}_{i}''^-$ is 3-chromatic.
		
		To that end, it is enough to show that each of them is 3-colorable.
		By symmetry, suppose $\widehat{H}_{i}'$ is not 3-colorable. Let $A'_i$ be the subset of vertices in $A'$ which are adjacent to every vertex in $H_i$. Applying \Cref{lem:FullJoin} to the subgraph induced by $H_{i}$ and $A'$, we conclude that the set $A_i'\cup V(H_{i}'')$ induces no negative edge. We now consider $A''=A'\cup V(G_{i}'')$. As we observed, $A''$ induces no negative edge. The components of $G\setminus A''$ are either $\widehat H_j$, $j\neq i$, or the components of $H_{i}'$. Thus each of them is a module in $G$. But this contradicts the maximality of $A'$, which proves our claim.

		Next, we claim that that there is no negative edge connecting $A_i'$ to an $\widehat H_i$.
		Consider an arbitrary vertex $x$ of $A'$. Since $G$ induces no $(K_3,-)$, the negative neighborhood of $x$ in $\widehat H_i$, denoted $N_i^-(x)$, is an independent set of $\widehat{H}_{i-}$. Thus $\widehat{H}_{i-} \setminus N_i^-(x)$ is of chromatic number $5$ or $6$. In either case at least one of the two subgraphs, $\widehat{H}_{i-}'\setminus N_i^-(x)$ or $\widehat{H}_{i-}''\setminus N_i^-(x)$, is of chromatic number 3, that is to say it induces an odd cycle. Suppose $\widehat{H}_{i-}' \setminus N_i^-(x)$ induces an odd cycle $C$ and let $v$ be a vertex in $V(\widehat{H}_{i-}'')\cap N_i^-(x)$ (assuming it exists). Since there is no $(K_3,-)$, $N_i^-(v) \cap C$ is an independent set and since $C$ is odd, there is an edge $uw$ of $C$ which has no negative connection to $v$. Then those edges must be positive because $v$ is fully joined to $C$. But then the vertices $x$, $v$, $u$, and $w$ induce a $(K_4,M)$, we have a contradiction, which is shown in \Cref{fig:K4_P4_locally_bounded}. This implies that $x$ has no negative edge connecting it to $\widehat{H}_{i-}''$, but then noting that $\chi(\widehat{H}_{i-}'')=3$ and taking $C$ to be an odd cycle in $\widehat{H}_{i-}''$, we can repeat the same argument to conclude that $x$ has no negative connection to $\widehat{H}_{i-}'$ either. 
		
		In conclusion, the negative edges of each component of $\widehat{G}\setminus A'$ either induces a 5-colorable graph in which case we use colors $1,2, 3, 4,5$ to color it, or it induces a 6-chromatic graph with no negative edge connecting it to $A'$. In this case we use colors $1,2,3,4,5,6$ to color it. Finally we can use color 6 on the vertices of $A'$, concluding with a 6-coloring of $\widehat{G}^-$.
		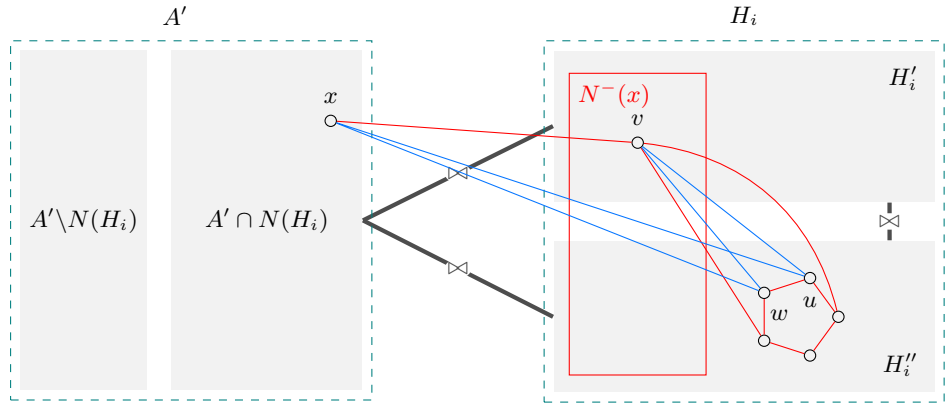
\begin{figure}[h!] \centering 
			\begin{tikzpicture}[scale=0.6, box/.style={draw=none, fill=gray!10, thick}, every node/.style={font=\small} ]

\node[box, minimum width=1.6cm, minimum height=4.5cm] (A1) { $A'\backslash N(H_i)$ }; 
\node[box, minimum width=2.5cm, minimum height=4.5cm, right=0.3cm of A1] (A2) { $A'\cap N(H_i)$};

\node[box, minimum width=5cm, minimum height=2cm, right=2.5cm of A2.north east, anchor=north west] (H1) {};
\node[box, minimum width=5cm, minimum height=2cm, below=0.5cm of H1] (H2) {};
;

\node[anchor=north east, xshift=-0.1cm, yshift=-0.1cm] at (H1.north east) {$H'_i$};

\node[anchor=south east, xshift=-0.1cm, yshift=0.1cm] at (H2.south east) {$H''_i$};

\draw[ultra thick, black!70!white, rounded corners=3pt] 
    (A2.east) -- (H2.west)
    node[midway, fill=white, inner sep=1pt] {$\bowtie$};
\draw[ultra thick, black!70!white, rounded corners=3pt] 
    (A2.east) -- (H1.west)
    node[midway, fill=white, inner sep=1pt] {$\bowtie$};
   \draw[ultra thick, black!70!white, rounded corners=3pt] 
    ([xshift=-1cm]H1.south east) -- ([xshift=-1cm]H2.north east)
    node[midway, fill=white, inner sep=1pt] {$\bowtie$}; 

\node[draw=teal,dashed,fit=(A1)(A2)] {}; 
\node[draw=teal,dashed,fit=(H1)(H2)] {};

\node[above=0.2cm] at ($(A1.north)!0.5!(A2.north)$) {$A'$};
\node[above=0.2cm of H1] {$H_i$}; 

\node[draw=red, fill=none, rectangle, minimum width=1.8cm, minimum height=4cm, xshift=-1.4cm, yshift=-1.3cm] (RedBlock) at (H1){};

\node[anchor=north west, xshift=0mm, yshift=0mm] at (RedBlock.north west)  {{\textcolor{red}{$N^-(x)$}}};

\node[circle, draw=black, fill=none, inner sep=1.5pt] (v) at ([xshift=0cm,yshift=1.8cm]RedBlock) {};
\node[above=1pt of v] {$v$};

\node[circle, draw=black, fill=none, inner sep=1.5pt] (x) at ([xshift=1.4cm,yshift=2.2cm]A2) {};
\node[above=1pt of x] {$x$};

\node[draw=none, fill=none, thick, ellipse,
      minimum width=3.5cm, minimum height=1.5cm,
      xshift=0.7cm] (BlueBlock) at (H2){};
      \node[anchor=south east, xshift=4mm, yshift=-2mm] at (BlueBlock.south east) {\textcolor{blue}{}};

\foreach \i in {1,...,5} {
  \node[circle, draw=black, fill=none, inner sep=1.5pt] (c\i) 
        at ([shift={(72*\i:0.9cm)}] BlueBlock) {};
}
\foreach \i [evaluate=\i as \j using {int(mod(\i,5)+1)}] in {1,...,5} {
  \draw[edge red] (c\i) -- (c\j);
}
\node[below=1pt of c1] {$u$};
\node[below=0pt of c2, xshift=2mm, yshift=-0.2mm] {$w$};

\draw[edge red](x)--(v);
\draw[edge red](c3)--(v);
\draw[blue_edge](c1)--(v)--(c2);
\draw[blue_edge](c1)--(x)--(c2);
\draw[edge red, bend left=35] (v) to (c5);

\end{tikzpicture} 
			\caption{Illustration of the proof of \Cref{thm:K4_P4_locally_bounded}} 
			\label{fig:K4_P4_locally_bounded} 
		\end{figure}		
	\end{proof}

	Applying \Cref{prop:SwitchVSColored} we have the following.
	
	\begin{corollary}\label{cor:K4P4-free}
		The class $\F\{\widehat{(K_4,-)}, P_4\}$ is a GS set.
	\end{corollary}

	Combining this corollary with \Cref{lem:magic}, we have the following conclusion.
	
	\begin{theorem}\label{thm:linearforset}
		If each of the connected components of a linear forest $F$ has at most 4 vertices, then $\F\{\widehat{(K_4,-)}, F\}$ is a GS set. 
	\end{theorem}

	For each of the three results \Cref{thm:K4_P4_locally_bounded}, \Cref{cor:K4P4-free}, \Cref{thm:linearforset}, we do not know what is the best possible upper bound. In \Cref{thm:K4_P4_locally_bounded} we provided an upper bound of 6, but in fact this was the upper bound on the chromatic number of the subgraph induced by the negative edges. While we do not know if 6 is the best bound for the balanced chromatic number of the family, in the following we provide an example for which the negative edges induces a 6-chromatic graph. Applying this upper bound of 6, for  \Cref{cor:K4P4-free} we get an upper bound $2^7$. For \Cref{thm:linearforset}, the upper bound is a function of the number of components of $F$. 
	
	\begin{proposition}
		There is a signed graph in $\F((K_3,-), (K_4,M), P_4)$ whose negative edges induce a graph of chromatic number 6.
	\end{proposition}
	
	\begin{proof}
		The basic element of the construction is the signed graph of \Cref{fig:envelope} called \emph{envelope} and denoted $\widehat{EN}$.

		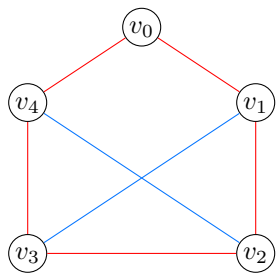
\begin{figure}[h!]
			\centering
			\begin{tikzpicture}[scale=0.5, baseline=(current bounding box.center),
				node style/.style={
					circle, draw,
					inner sep=0pt,
					minimum size=5mm,
					text height=1.2ex, text depth=.2ex 
				}
				]
				
				\node[node style] (P0) at (3,6){$v_0$};
				\node[node style] (P1) at (6,4){$v_1$};
				\node[node style] (P2) at (6,0){$v_2$};
				\node[node style] (P3) at (0,0){$v_3$};
				\node[node style] (P4) at (0,4){$v_4$};

				\draw[red] (P0) -- (P1) -- (P2) -- (P3)--(P4)--(P0);
				\draw[blue_edge] (P2) -- (P4);
				\draw[blue_edge] (P1) -- (P3);
			\end{tikzpicture}
			\caption{The envelope}
			\label{fig:envelope}
		\end{figure}

		\begin{claim}\label{lem:XYZ_from_coloring}  Given an integer $c\geq 3$, if $k \geq 2c-3$, for any $c$-coloring $\phi$ of the disjoint union of $k$ 5-cycles, there exist three disjoint independent sets $X,Y$ and $Z$ such that $$ |\phi(X)\cap \phi(Y) \cap \phi(Z)| \geq 3.$$
		\end{claim}
		{\em Proof of the claim.}   
		We first claim that there are three colors, say $a$, $b$, $c$, such that each color appears in at least three of the $C_5$'s. If not, after removing the two colors that appear in the most $C_5$'s, each $C_5$, being 3-chromatic,  has a color that is shared with at most one other $C_5$. Thus the number of colors, $c$, is at least $\lceil \frac{k}{2}\rceil+2$, contradicting the assumption that $k \geq 2c-3$.
		
		Let $A$ be a set of three vertices from three distinct $C_5$ all colored $a$. Similarly we choose $B$ and $C$ each consisting of three vertices, colored $b$ and $c$ respectively. We now consider the incidence graph between $\set{A, B, C}$ and the $C_5$'s.
		Observe that this is a bipartite graph of maximum degree 3. Thus it admits a 3-edge-coloring. Let $X$, $Y$, $Z$ be the three color classes in one such coloring. One can view each of $X$, $Y$, and $Z$ as a 3-subset of the vertices of $C_5$'s, spanning the tree colors $a,b,c$ and coming from different $C_5$'s. Thus $X$, $Y$, $Z$ are disjoint independent sets and $|\phi(X)\cap \phi(Y) \cap \phi(Z)| \geq 3$. \hfill $\diamond$ \\    
		
		We want to construct a signed graph $\widehat{G}$ in ${\cal G} = \F((K_3,-), (K_4,M), P_4)$ such that the subgraph $\widehat{G}_{-}$ has chromatic number $6$. To that end we work with red-independent sets that are independent sets of $\widehat{G}_{-}$. 
		First observe that the envelope is in $\cal G$. Start with the signed graph $\widehat{R}$ which is the union of 7 disjoint envelopes. 
		Construct the signed graph $\widehat{L}$ as follows. For every possible $5$-coloring $\phi$ of $\widehat{R}_{-}$, add a disjoint envelope ${\widehat{EN}}_{\phi}$ in $\widehat{L}$ fully joined with $\widehat{R}$.
		Applying \Cref{lem:XYZ_from_coloring} to $\widehat{R}_{-}$, we find disjoint independent sets $X,Y,Z$ (of $\widehat{R}_{-}$). As $\widehat{R}_{-}$ is a union of disjoint $C_5$'s, hence 2-regular, each vertex not in $X\cup Y \cup Z$ can be assigned to one of the sets while they remain independent sets. Thus we may assume the $X\cup Y \cup Z=V(\widehat{R}_{-})$, and note that we still  have the property $|\phi(X) \cap \phi(Y) \cap \phi(Z)| \geq 3$.  
		
		The signs of the edges in the full join of ${\widehat{EN}}_{\phi}$ to $\widehat{R}$ are chosen as follows. The copy of $v_0$ in ${\widehat{EN}}_{\phi}$ is adjacent to every vertex in $X$ by a negative edge, copies of $v_1,v_3$ in ${\widehat{EN}}_{\phi}$ are both adjacent to every vertex in $Y$ by negative edges, and copies of $v_2,v_4$ in ${\widehat{EN}}_{\phi}$ are both adjacent to all vertices in $Z$ by negative edges. All other edges in the full join are positive.
		The resulting signed graph is called $\widehat{LR}$. 
		
		\begin{claim}
			$\widehat{LR} \in \cal G$.
		\end{claim}

		The underlying graph is obtain from a full join of disjoint unions of envelopes which are cographs. Therefore, it is itself a cograph, i.e. $P_4$-free.
		For every vertex $u$ in $V(\widehat{L})$, $N_{R}^-(u)$ does not induce a negative edge because $X$ (resp $Y$,$Z$) is an independent set in $\widehat{R}_{-}$.
		For every pair $u,v$ of vertices in $\widehat{L}$ connected by a negative edge, we have 1. $N_{R}^-(u)\cap N_{R}^-(v)= \emptyset$ 2. $N^+(u)\cap N^+(v)$ contains no negative edge. The former is simply because the three sets $X$, $Y$, and $Z$ are disjoint. The latter is because $N^+(u)\cap N^+(v) = R\setminus(N^-(u)\cup N^-(v))$ which is one of the $X,Y,Z$ and hence is an independent set in $\widehat{R}_{-}$.
		
		Those conditions, together with the fact that there is no triangle with two positive edge in an envelope, we are ensured that there is neither $(K_3,-)$ nor $(K_4,M)$ in $\widehat{LR}$.
		
		\begin{claim}
			The graph $\widehat{LR}_{-}$ has chromatic number 6.
		\end{claim}

		Suppose $\phi$ is a $5$-coloring of $\widehat{LR}_{-}$. Let $\phi_R$ be the coloring induced on $\widehat{R}_{-}$ and consider $\widehat{EN}_{\phi_R} \in \widehat{L}$ and the partition $X_{\phi},Y_{\phi},Z_{\phi}$ of $V(\widehat{R})$ associated to $\phi$.
		Since $|\phi(X) \cap \phi(Y) \cap \phi(Z)| \geq 3$, there are at least 3 colors not used on $\phi(\widehat{EN}_{\phi_R})$. But $\widehat{EN}_{\phi_R-}$ (which is a $C_5$) needs at least three more colors, which is a contradiction.
	\end{proof}
	
	\section{Conclusion}
	
	In this work, we initiated the study of balanced chromatic number on the hereditary classes of signed graphs. 
	
	The notion of balanced chromatic number generalizes the classic chromatic number of graphs. Given a graph $G$, the balanced chromatic number of the signed graph $\widetilde{G}$ is the same as the chromatic number of $G$. The class of all signed graphs of the form $\widetilde{G}$ is the hereditary class $\F{\{\widehat{(K_2,-)}\}}$. That is the class of signed graphs where no pair of vertices induces a simple edge.
	Thus the \Gya-Summer conjecture can be restated in the language of signed graphs as follows.
	
	\begin{GSconjecture}[Restated] Any minimal (finite) GS set containing $\widehat{(K_2,-)}$ has three elements.		
	\end{GSconjecture}	
	
	In this work, then, we studied minimal (finite) GS sets that do not contain $\widetilde{K}_2$. It is observed that any such set must be of order at least 3 (including $\widetilde{K}_2$). Our focus has been to classify those sets of order 3. We showed that any such set must contain either $\widehat{(K_3,-)}$ or $\widehat{(K_4,-)}$ and a linear forest. For $\widehat{(K_3,-)}$ we showed that all such sets are GS sets. In the case of $\widehat{(K_4,-)}$ we showed that, as long as each component of the linear forest is of order at most 4, we have a GS set. We conjectured that for every linear forest $F$ the set $\{\widetilde{K}_2, \widehat{(K_4,-)}, F\}$ is a GS set. 
	  
	Thus, unlike the case where $\widehat{(K_2, -)}$ is forbidden, in the case where $\widetilde{K}_2$ is forbidden, we already have minimal GS sets of order at least 4, and perhaps we can build minimal GS sets of larger size. In a follow-up work we study minimal GS sets of order 4.

	\section*{Acknowledgment} This work has received support under the program ``Investissement d'Avenir" launched by the French Government and implemented by ANR, with the reference ``ANR‐18‐IdEx‐0001" as part of its program ``Emergence".
	
	\bibliographystyle{plain}
	\bibliography{RefencesGS}	
	
\end{document}